\documentclass[11pt,A4]{amsart}
\usepackage{geometry,graphicx,fancybox}
\usepackage{latexsym}
\usepackage{labelfig}
\usepackage{amssymb}
\usepackage[cp850]{inputenc}
\usepackage{epsfig}
\usepackage{epstopdf}
\usepackage{psfrag}
\usepackage{amsthm}
\usepackage{amscd}
\usepackage{amsmath}
\usepackage{amsfonts}
\usepackage{graphics}
\usepackage{enumerate}

\newtheorem{theorem}{Theorem}[section]
\newtheorem{proposition}[theorem]{Proposition}
\newtheorem{lemma}[theorem]{Lemma}

\theoremstyle{definition}

\newtheorem{remark}[theorem]{Remark}

\newcommand{\sys}{{\rm sys}}
\newcommand{\ii}{{\mathit i}}

\newcommand{\R}{{\mathbb R}}

\newcommand{\Z}{{\mathbb Z}}

\newcommand{\M}{{\mathcal M}}
\newcommand{\PSL}{{\rm PSL}}

\newcommand{\arcsinh}{{\,\rm arcsinh}}
\newcommand{\arccosh}{{\,\rm arccosh}}

\long\def\forget#1\forgotten{}

\numberwithin{equation}{section}

\begin{document}

\forget
Nothing appears in the document !
\forgotten

%




%

\title{The homology systole of hyperbolic Riemann surfaces}




%

\author[H.~Parlier]{Hugo Parlier}
\address[Hugo Parlier]{Department of Mathematics, University of Fribourg\\
  Switzerland}
\email{hugo.parlier@gmail.com}
\thanks{Research supported by Swiss National Science Foundation grant number PP00P2\textunderscore 128557}

\date{\today}



\begin{abstract} 
The main goal of this note is to show that the study of closed hyperbolic surfaces with maximum length systole is in fact the study of surfaces with maximum length homological systole. The same result is shown to be true for once-punctured surfaces, and is shown to fail for surfaces with a large number of cusps. 
\end{abstract}




%



\maketitle


\section{Introduction}

There is a natural function on moduli space, called the systole function, which associates to a hyperbolic surface the length of its shortest non-trivial closed curve. Unless the hyperbolic surface is topologically a pair of pants, the curve in question is a simple closed geodesic generally called the systole or the systolic loop. Although there exist surfaces with arbitrarily small systole, the systole function is bounded over any given moduli space (of complete finite area surfaces of a given signature). Furthermore via Mumford's compactness theorem, there is (at least) a surface in each moduli space which realizes the maximum length systole. The study of these surfaces, and more generally the study of critical points of the systole function, largely initiated by Schmutz Schaller \cite{sc931}, has generated quite a bit of interest and can be thought of as a type of hyperbolic sphere packing problem. Notable results include a sharp upper bound on the systole among all surfaces of genus $2$ \cite{je84}, the fact that principal congruence subgroups of $\PSL_2(\Z)$ give rise to global maxima in their respective signatures \cite{ad98, sc941} and Akrout's theorem that the systole function is in fact a topological Morse function \cite{ak03} (partial results previously due to Schmutz Schaller \cite{sc99}).

If we define $\sys_{g}$  (or more generally $\sys_{g,n}$) to be the maximum length of a systole among all hyperbolic surfaces of genus $g$ (resp. of genus $g$ with $n$ cusps), it is also an interesting question to ask how these constants grow as a function of topology. As it turns out, the more interesting question is how these constants grow as a function of genus, because for large enough $n$, they cease to grow \cite{ad98,sc941}. Note that the systole length of a closed hyperbolic surface is exactly twice the length of the minimum injectivity radius. Thus if a surface has a systole of length $\ell$, then around any point of the surface, there is an embedded open disk of radius $\ell\over 2$. By considering the area of a disk around a point in the hyperbolic plane, which grows roughly exponentially in radius, one immediately sees that the systole function is bounded by roughly $\log(g)$. In the more general case of closed Riemannian surfaces with area normalized to $4\pi(g-1)$, the result is also true and this is a theorem of Gromov \cite{gr83}. Conversely, there are constructions of families of surfaces, one in each genus, where the systoles grow roughly like $\log(g)$. The first of these constructions was due to Buser and Sarnak \cite{busa94}, and there have been others since \cite{br99,kascvi07}. Both Buser-Sarnak, and Gromov  in the more general setting of Riemannian metrics, also considered the homological systole $\sys^h(S)$, i.e., the shortest homologically non-trivial curve on a surface $S$ of genus $g$. Of course one has the obvious inequality $\sys(S)\leq \sys^h(S)$ for any surface $S$ and it is easy to construct surfaces where the inequality is an equality, resp. where the inequality is strict. Notice that the embedded disk argument above does not necessarily give a homologically non-trivial curve. Nonetheless, as in the case of the homotopy systole, it is not too difficult to find a rough $\log(g)$ upper bound on the homology systole of a hyperbolic surface of genus $g$ \cite{busa94}. Again, the rough $\log(g)$ bound remains true in the setting of Riemannian metrics of normalized area \cite[2.C]{gr96}. 

The goal of this note is to observe that in any signature, there are surfaces which realize the supremum of the homological systole function $\sys_{g,n}^h$ ( $\sys^h_{g,0}=\sys^h_g$ in the closed case) and that in the closed and once-punctured cases, these are the same surfaces that realize the maximum homotopy systole.  Specifically, in section \ref{sec:proof}, the following is shown.

\begin{theorem}\label{thm:main} If $S$ is maximal for $\sys$ among all closed genus $g$ (resp. genus $g$ with one cusp) hyperbolic surfaces, then it is maximal for $\sys^h$. Thus $\sys_g= \sys^h_g$ and $\sys_{g,1}= \sys^h_{g,1}$.
\end{theorem}

However, it is clear that one cannot hope to generalize the above result to arbitrary signature. Indeed, via Buser's hairy torus examples \cite{buhab, bubook}, one can construct a family of surfaces of genus $1$ with $n$ cusps with homology systoles of length roughly $\sqrt{n}$. In contrast, the homotopy systole of a surface of genus $1$, no matter how many cusps it has, is uniformly bounded. In section \ref{sec:mul}, this example is adapted to arbitrary genus to show $\sys^h_{g,n}  > \sys_{g,n}$ for $n\geq 25 g$.

\section{Proof of Theorem \ref{thm:main}}\label{sec:proof}

We'll begin by showing that like the usual systole, the homological systole admits a maximum over the moduli space of genus $g$ hyperbolic surfaces with $n$ cusps. 
For the systole function, this is an immediate consequence of the continuity of the systole function and Mumford's compactness theorem \cite{mu71} which states the set of surfaces with injectivity radius bounded below is a compact subset of moduli space. Here we need to be more careful because a priori surfaces could be arbitrarily close to the supremum of the homological systole and have a an arbitrarily small systole (which in this case would be homologically trivial). The following lemma will allow us to again apply Mumford's compactness theorem. Before stating the lemma, we set $\M_{g,n}^{\varepsilon}$ to be the set of all surfaces of genus $g$ with systole bounded below by $\varepsilon$ and $\left(\M_{g,n}^{\varepsilon}\right)^c$ to be the set of all surfaces with systole strictly less than $\varepsilon$.

\begin{lemma}\label{lem:max} For each signature $(g,n)$, there is an $\varepsilon_{g,n}>0$ such that 
$$
\sup_{S\in \left(\M_{g,n}^{\varepsilon_{g,n}}\right)^c} \sys^h(S) < \sup_{S\in \M_{g,n}^{\varepsilon_{g,n}}} \sys^h(S).
$$
\end{lemma}

Before giving the proof of the lemma, we recall the following lemma (see for instance \cite{path101,pa051,thspine}).

\begin{lemma}\label{lem:lel}[Length expansion lemma] Let $S$ be a surface with $N > 0$ disjoint simple closed geodesics $\gamma_1,\hdots,\gamma_N$ of lengths $\ell_1,\hdots,\ell_N$. For
$(\delta_1,\hdots,\delta_N) \in (\R^+)^N$ with at least
one $\delta_k\neq 0$, there exists a surface $\tilde{S}$ with $\ell_{\tilde{S}}(\gamma_1)=\ell_1+\delta_1,\hdots,\ell_{\tilde{S}}(\gamma_N)=\ell_N+\delta_N$ and all simple closed geodesics of $\tilde{S}$ disjoint from $\gamma_1,\hdots,\gamma_N$ of 
length strictly greater than their length on $S$.
\end{lemma}

\begin{proof}[Proof of lemma \ref{lem:max}]
Set 
$$
\varepsilon'_{g,n}= \frac{1}{2} \sup_{S\in \M_{g,n}} \sys^h(S)= \frac{1}{2}\sys^h_{g,n}
$$
and consider a surfaces in $\left(\M_{g,n}^{\varepsilon'_{g,n}}\right)^c$. Clearly, the supremum of $\sys^h$ among all such surfaces with homological systole less than $\varepsilon_{g,n}'$ will be strictly less than $\sys^h_{g,n}$. We only need to worry about surfaces with homologically trivial systole in $\left(\M_{g,n}^{\varepsilon'_{g,n}}\right)^c$. Let $S$ be such a surface.

The collar lemma \cite{kee74} ensures that if a simple closed geodesic $\gamma\subset S$ is sufficiently short, then any simple closed geodesic that intersects it transversally is long (where long depends only on how short the geodesic is). In particular, as $\sys^h$ is bounded above over $\M_{g,n}$, if $\gamma$ is less than a certain constant $\varepsilon_{g,n}^{\prime \prime}$ (which depends only on the topology) then all of the surface's homological systoles are disjoint from $\gamma$. We can now apply lemma \ref{lem:lel} to increase the (homologically trivial) systoles of $S$ by some small $\delta>0$ which will give us a new surface with $\sys^h(S') > \sys^h(S)$ and $\sys(S')= \sys(S)+\delta$. This process can be repeated until either the systole of the surface is now $\varepsilon'_{g,n}$, or there is a homological systole which crosses each of the homologically trivial systoles. In the latter case, both the homologically trivial and non-trivial systoles are at least of length equal to the constant $\varepsilon_{g,n}^{\prime \prime}$. We set 
$$
\varepsilon_{g,n} = \min \left( \varepsilon_{g,n}', \varepsilon_{g,n}^{\prime \prime}\right)$$
and the lemma is proved. 
\end{proof}

We've shown that $\sup_{S\in \M_{g,n}} (\sys_h(S))= \sup_{S\in \M_{g,n}^{\varepsilon_{g,n}}} (\sys_h(S))$ thus via Mumford's compactness theorem, we can now conclude that there are  surfaces which realize the maximum size homological systole in every signature.

A closed geodesic $\gamma$ is said to be {\it straight} if for any two of its points $p,q$, any distance realizing path is a sub-arc of $\gamma$. Homological systoles {\it always} have this property \cite{gr83}. The following is well known for closed surfaces, and we provide a proof that includes once-punctured surfaces.

\begin{lemma}\label{lem:straight} On closed and once-punctured surfaces, systoles are straight. 
\end{lemma}
\begin{proof}[Proof of lemma \ref{lem:straight}]
Suppose that a systole $\sigma$ is not straight, thus there is a path $c$ between two of its points $p$ and $q$ of length less than the length of the shortest of the two paths $c_1$ and $c_2$ of $\sigma$ between the same points. 

\begin{figure}[h]
\leavevmode \SetLabels
\L(.34*.58) $c$\\
\L(.3*.27) $c_1$\\
\L(.4*.27) $c_2$\\
\L(.67*.6) $c$\\
\L(.60*.35) $c_1$\\
\L(.53*.34) $c_2$\\
\endSetLabels
\begin{center}
\AffixLabels{\centerline{\epsfig{file =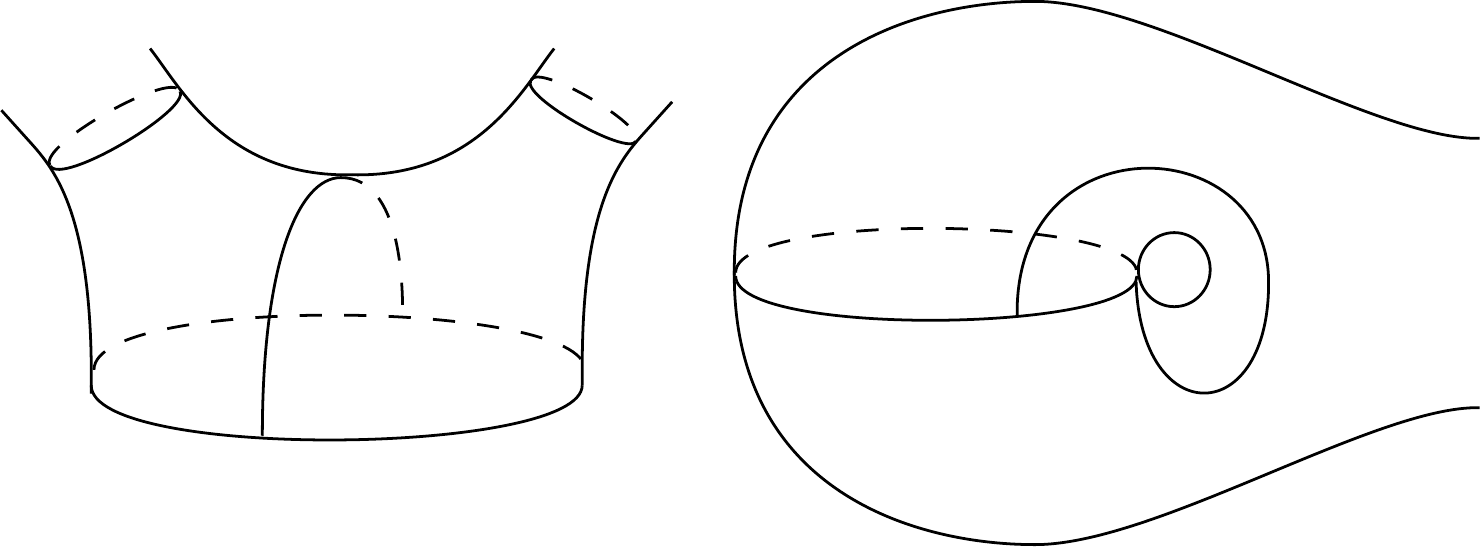,width=8cm,angle=0}}}
\end{center}
\caption{The two cases for $c,c_1,c_2$} \label{fig1}
\end{figure}
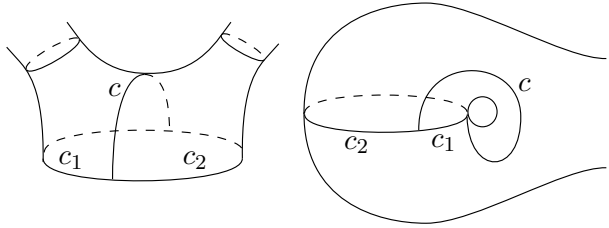

Consider the two homotopy classes $c\cup c_1$ and $c\cup c_2$. They are either non-trivial or are parallel to a cusp. If the surface has at most one cusp, one of them, say $c\cup c_1$ must be non-trivial. The geodesic $\tilde{\sigma}$ in the homotopy class satisfies $\ell(\tilde{\sigma}) < \ell(c)+\ell(c_1) \leq \ell(\sigma)$, a contradiction.
\end{proof}

\begin{remark} The lemma above cannot be generalized to surfaces with multiple punctures. To see this, consider the following twice-punctured surface which is constructed as follows. 

We begin with any finite trivalent graph $G$ of girth at least $4$ (the girth is the shortest non-trivial cycle). (An example of such a graph is the complete bipartite $K_{3,3}$ graph: in fact, it is the smallest trivalent graph with girth $4$ and it has exactly $6$ vertices.) A new graph $\tilde{G}$ is obtained by removing a single edge and replacing it with a tripod (see figure \ref{fig2}). This operation does not decrease the girth of the graph. 

\begin{figure}[h]

\begin{center}
\AffixLabels{\centerline{\epsfig{file =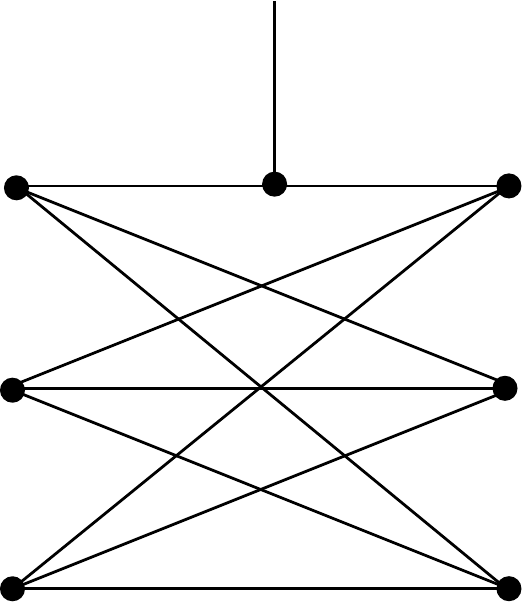,width=4cm,angle=0}}}
\end{center}
\caption{The modified $K_{33}$ graph} \label{fig2}
\end{figure}

We view the resulting graph as the graph of a pants decomposition, and construct a surface by inserting (hyperbolic) pants with boundary lengths all equal to $4\arcsinh 1$ in the usual way, without paying any attention to twist parameters. The resulting surface $B$ has a single boundary geodesic of length $4\arcsinh 1$. To obtain a surface with $2$ cusps, we glue a pair of pants with $2$ cusps and a boundary geodesic of length $4\arcsinh 1$ along the boundary geodesic of $B$. As an illustration, if we choose the $K_{3,3}$ graph to begin with, the resulting surface has a total of $8$ pairs of pants and is of genus $4$. 

Our claim is that the systoles of this surface are exactly the $N$ ($=11$ in the $K_{3,3}$ case) geodesics of length $4\arcsinh 1$ which were the boundary geodesics of the pants. We'll denote these curves $\sigma_1,\hdots,\sigma_{N}$ with $\sigma_1$ being the geodesic that forms a pair of pants with the two cusps.

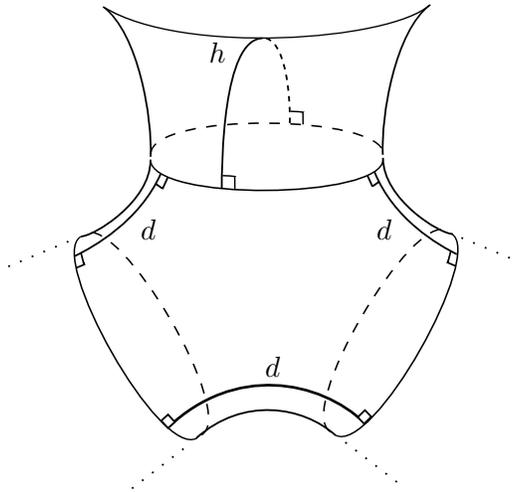
\begin{figure}[h]
\leavevmode \SetLabels
\L(.455*.88) $h$\\
\L(.4*.52) $d$\\
\L(.5*.24) $d$\\
\L(.59*.52) $d$\\
\endSetLabels
\begin{center}
\AffixLabels{\centerline{\epsfig{file =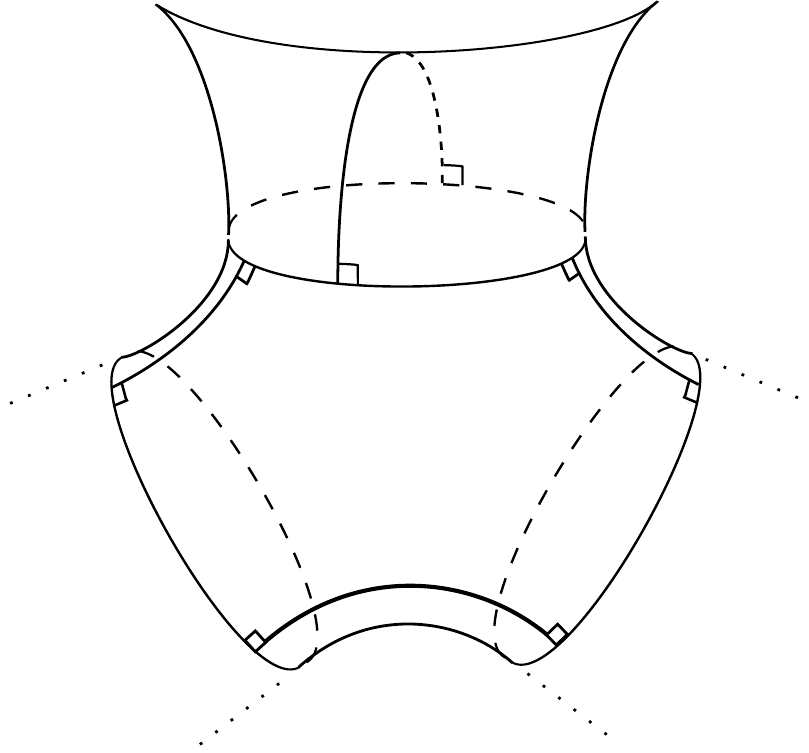,width=7cm,angle=0}}}
\end{center}
\caption{Adding a cusped pair of pants} \label{fig3}
\end{figure}

To see that these curves are indeed systoles, we will show that any geodesic that intersects one of these curves is strictly longer. We begin by taking any curve $\gamma$ that crosses one of the curves $\sigma_k$ and has an arc $c$ that leaves and comes back on the same side of $\sigma_k$. (In the event where this curve is $\sigma_1$ we consider the arc of $\gamma$ not contained in the pair of pants with $2$ cusps.) We replace $c$ with the unique shortest geodesic $\tilde{c}$ in the free homotopy class where the endpoints are allowed to slide on $\sigma_k$. Now $\tilde{c}$ separates $\sigma_k$ into two arcs, one of which, say $\hat{c}$, has length less than $2 \arcsinh 1$. Via standard hyperbolic trigonometric arguments, the unique geodesic in the homotopy class of $\tilde{c}\cup \hat{c}$ has length less than $\ell(\tilde{c})$, thus less than $\ell(c)$, and less than $\ell(\gamma)$. It follows that we can restrict our attention to curves that don't have this ``backtracking" property.

Such curves describe non-trivial cycles in the underlying graph, and have cycle length at least $4$. It follows that they pass through at least $4$ pairs of pants and as a consequence, their length is at least $4$ times the shortest distance between boundary curves of the pants. Again, via standard hyperbolic trigonometry, this distance is $d= 2 \arcsinh(\frac{1}{2})$. These curves are thus at least of length $8 \arcsinh(\frac{1}{2}) > 4 \arcsinh 1$.

We conclude by observing that although the curve $\sigma_1$ is a systole, it is not straight. Indeed, on the pair of pants with $\sigma_1$ and the two cusps, there is a unique simple geodesic path $h$, perpendicular to $\sigma_1$ in both endpoints. Again, via hyperbolic trigonometry its length is exactly $2 \arcsinh 1$. It is thus a distance realizing path between the endpoints and is not contained in $\sigma_1$.
\end{remark}

We now use lemma \ref{lem:straight} to show the following.

\begin{lemma}\label{lem:int} Let $\gamma$ be the homological systole of a surface $S$ with at most one puncture. Then if $\delta$ is the systole of $S$, then $\gamma$ and $\delta$ intersect at most once.
\end{lemma}
\begin{proof}[Proof of lemma \ref{lem:int}]
We proceed by contradiction. Suppose $\gamma$ and $\delta$ as above intersect more than once. If we cut the surface along $\gamma$ then $\delta$ is cut into at least $2$ arcs. One of these arcs, say $c$, has the property of being of length less than $\frac{1}{2} \ell(\delta)\leq \frac{1}{2} \ell(\gamma).$ This arc $c$ can be of two types: either $c$ has its endpoints on the two copies of $\gamma$, or it has both endpoints on the same copy. 

In the first case, we consider the shortest arc $\tilde{c}$ of $\gamma$ between the two endpoints of $c$. Clearly $\ell(\tilde{c}) < \ell(\gamma)$. Note that $c \cup \tilde{c}$ describe a simple closed path, and consider the geodesic $\tilde{\gamma}$ in the homotopy class of $c\cup \tilde{c}$. By construction we have $\ii(\gamma,\tilde{\gamma})=1$ and $\ell(\gamma)< \frac{1}{2} (\ell(\gamma) + \ell(\delta))\leq \ell(\gamma)$. Observe that a curve that essentially intersects another curve exactly once is not only homotopically non-trivial, but also homologically non-trivial (a separating curve essentially intersects any other curve at least twice). 

In the second case, consider as before the arc $c$, the shortest of the arcs of $\delta$ obtained by cutting along $\gamma$. Consider $c_1$ and $c_2$ the two arcs of $\gamma$ between the endpoints of $c$. Consider the two geodesics $\gamma_1$ and $\gamma_2$ in the homotopy class of respectively $c \cup c_1$ and $c \cup c_2$. Observe that 
$$
\ell(\gamma_1),\ell(\gamma_2) < \ell(\gamma).$$
In homology, one can orient the curves so that $\gamma$ can be written as the sum of $\gamma_1$ and $\gamma_2$. This is because the three are the boundary curves of a pair of pants. The boundary curves of a pair of pants, for a certain given orientation, form a multicurve which is trivial in homology. It follows that if $\gamma$ is non-trivial in homology, then $\gamma_1$ and $\gamma_2$ cannot both be trivial. As both are of length strictly less than $\gamma$, we obtain a contradiction.
\end{proof}
 
We have the obvious inequalities $\sys_{g,n} \leq \sys^h_{g,n}$. We will now proceed to show that $\sys_g \geq \sys^h_g$, resp. $\sys_{g,1} \geq \sys^h_{g,1}$. Consider a maximal surface $S_{\max}$ for the homology systole in $\M_{g,n}$ for $n=0$ or $n=1$.\\

\noindent{\bf Claim.} {\it All systoles of $S_{\max}$ are homologically non-trivial.}

\begin{remark}
Note that the claim implies the desired inequalities for $n=0,1$:
$$\sys^h_{g,n} = \sys^h(S_{\max})=\sys(S_{\max})\leq \sys_{g,n}$$
and this concludes the proof of Theorem \ref{thm:main}.
\end{remark}
\begin{proof}[Proof of claim:]
We shall proceed by contradiction. Suppose $S_{\max}$ has a homologically trivial systole, i.e., a separating systole $\delta$. As $S_{\max}$ is maximal, the curve $\delta$ must cross a homological systole $\gamma$, otherwise by the length expansion lemma, one can increase $\delta$ to strictly increase the length of all simple closed geodesics that do not cross $\delta$, thus including all homological systoles. Now $\delta$ is separating and essentially intersects $\gamma$, thus $\delta$ must intersect $\gamma$ at least twice which by lemma \ref{lem:int} is a contradiction.
\end{proof}

\section{Surfaces with punctures}\label{sec:mul}

One could ask what happens for multiply punctured surfaces. What fails in the proof is that for $n\geq 2$, the systole of a surface is no longer necessarily straight. Note that the equality may in fact hold for $n=2$, only the method given here doesn't work. A general inequality of the form $\sys^h_{g,n}= \sys_{g,n}$ is deemed to fail however, as will be explained in what follows. 

The basic reason for this is that $\sys^h_{g,n}$ is a strictly growing function of $n$ and in contrast $\sys_{g,n}$ is uniformly bounded (by a function of $g$). One way of making this effective is by examining Buser's hairy torus examples (see \cite{bubook}). We recall briefly the construction and features of these surfaces.

One begins by constructing a hyperbolic ``square" (a right angled quadrilateral with equal length sides) with a cusp in the middle. Each side of the square can be taken to be of length $2 \arcsinh 1$. One then constructs an $m\times m$ checker board using $m^2$ of these squares. The board now has its $4$ sides of length $2 m \arcsinh 1$. A torus is then obtained by gluing the opposite sides in the obvious way. Now if one cuts along a non-homologically trivial curve on this surface, the genus must be reduced. Such a curve must cross either every horizontal line or every vertical line. As such, it must have length at least $2 m \arcsinh 1$.

\begin{figure}[h]
\begin{center}
\AffixLabels{\centerline{\epsfig{file =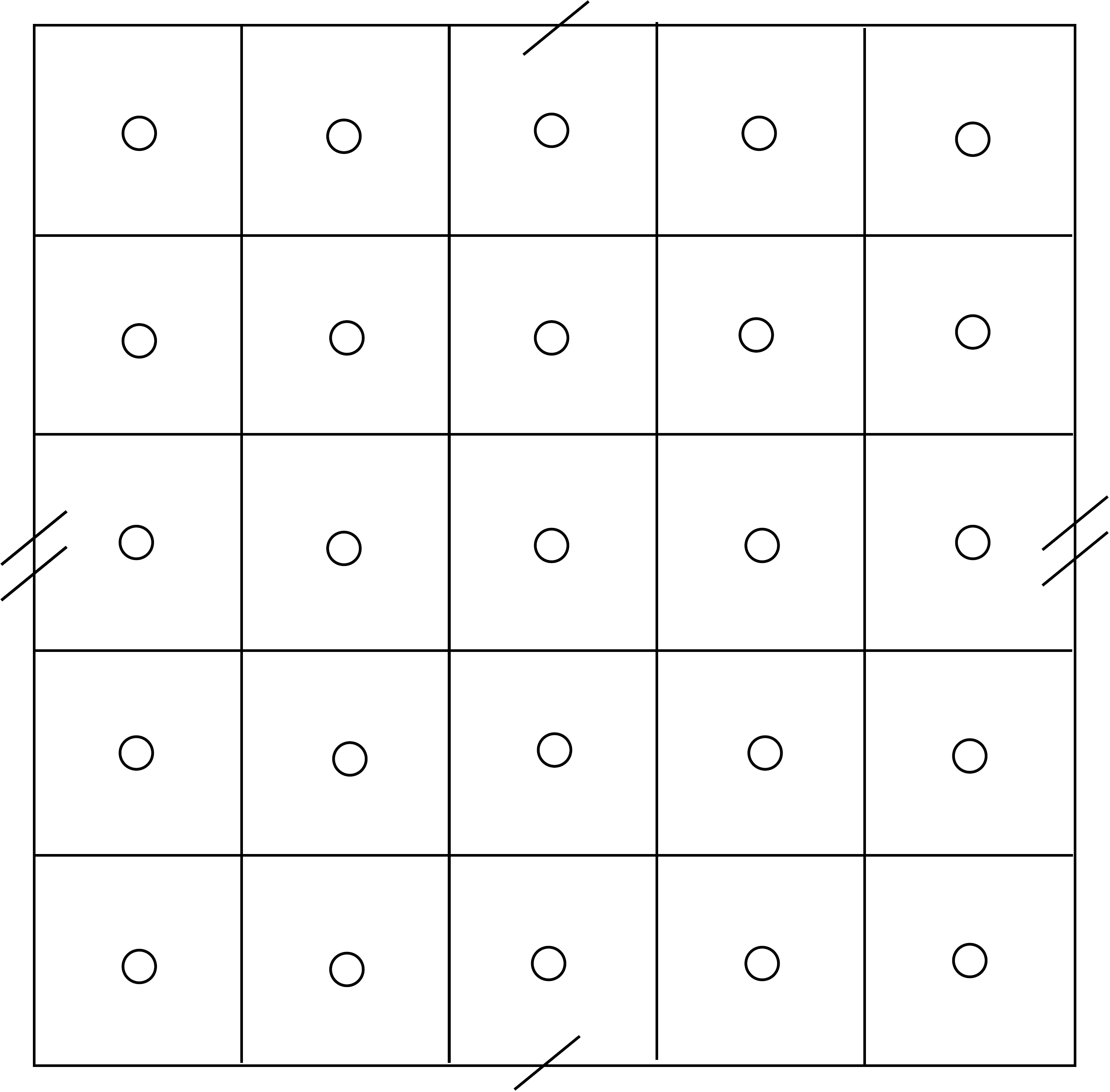,width=6cm,angle=0}}}
\end{center}
\caption{The schematics of a hairy torus - the circles represent cusps} \label{fig4}
\end{figure}

To obtain surfaces of genus $g$ surfaces which enjoy the same property, one can paste together $g$ copies of a hairy torus. To paste two tori together, one can replace a cusp on each torus by a very short geodesic in the standard way. The very short geodesic is homologically trivial, and again by the collar lemma, any curve that crosses it must be very long. In particular, a homology systole will not cross this curve and will thus remain in one of the two tori. One now repeats the construction to obtain a string of hairy tori (see figure \ref{fig5}). The resulting surface is of signature $(g, g m^2)$ with homology systole at least $2 m \arcsinh 1$. 

\begin{figure}[h]
\begin{center}
\AffixLabels{\centerline{\epsfig{file =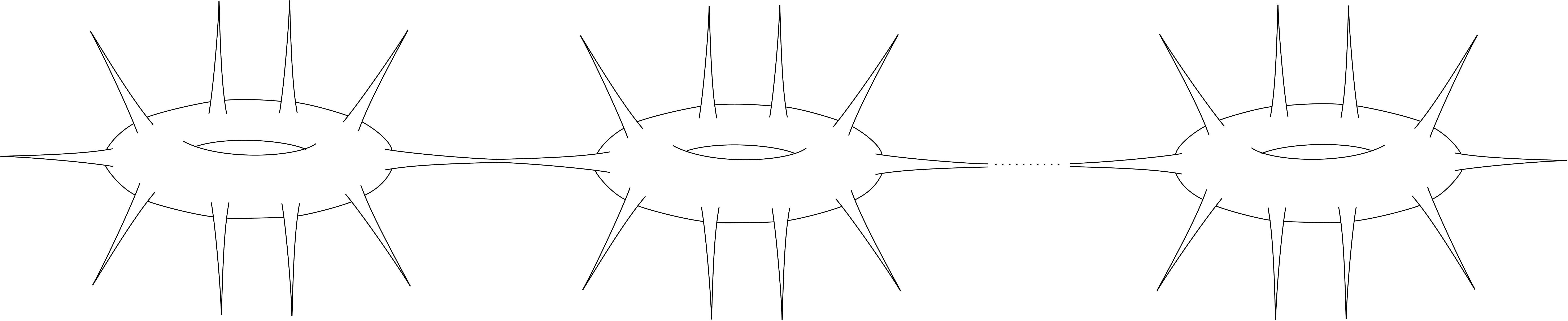,height=3cm,angle=0}}}
\end{center}
\caption{The hairy genus $g$ surface} \label{fig5}
\end{figure}

In contrast, note that for fixed genus, the systole length of a surface is uniformly bounded. More specifically, Schmutz Schaller \cite{sc941} proved that the systole of a surface of signature $(g,n)$ is bounded by $4 \arccosh ((6g - 6+3n)/n)$ (which is uniformly bounded by a constant which depends on $g$). For each $g$, we can now compute the minimal $m$ for which $2 m \arcsinh 1 > 4 \arccosh ((6g - 6+3n)/n)$ with $n= g m^2$. The first $m$ for which this occurs is $m=5$. As a result of this construction, and the monoticity of $\sys^h_{g,n}$ in $n$, we obtain the following.
\begin{proposition}
For all $n\geq 25 g$,\hspace{0.1cm}
$
 \sys^h_{g,n}  > \sys_{g,n}.
 $
\end{proposition}
The bound in the above proposition is certainly far from being sharp. It might be interesting to find for any given genus $g$ what the first $n_g$ is for which we have $\sys^h_{g,n_g}  > \sys_{g,n_g}$.

\bibliographystyle{amsalpha}
\def\cprime{$'$}
\providecommand{\bysame}{\leavevmode\hbox to3em{\hrulefill}\thinspace}
\providecommand{\MR}{\relax\ifhmode\unskip\space\fi MR }
\providecommand{\MRhref}[2]{%
  \href{http://www.ams.org/mathscinet-getitem?mr=#1}{#2}
}
\providecommand{\href}[2]{#2}

\end{document}